\documentclass[12pt]{amsart}

\usepackage{enumerate,amssymb,bbm}
\usepackage{verbatim}
\usepackage{fullpage}
\newtheorem{theorem}{Theorem}[section]
\newtheorem{corollary}[theorem]{Corollary}
\newtheorem{lemma}[theorem]{Lemma}
\newtheorem{proposition}[theorem]{Proposition}

\theoremstyle{definition}
\newtheorem{definition}[theorem]{Definition}
\theoremstyle{remark}
\newtheorem{remark}[theorem]{Remark}
\newtheorem{example}[theorem]{Example}

\newtheorem{notation}[theorem]{Notation}

\numberwithin{equation}{section}

\newcommand{\C}{\mathbb{C}}
\newcommand{\N}{\mathbb{N}}
\newcommand{\R}{\mathbb{R}}
\newcommand{\Z}{\mathbb{Z}}

\newcommand{\ve}{\varepsilon}

\newcommand{\lan}{\langle}
\newcommand{\ran}{\rangle}

\newcommand{\bA}{A}
\newcommand{\bB}{B}
\newcommand{\bC}{C}
\newcommand{\bL}{L}
\newcommand{\bX}{X}
\newcommand{\bY}{Y}
\newcommand{\bM}{M}
\newcommand{\bI}{\mathbf{I}}

\newcommand{\hA}{\hat{\bA}}
\newcommand{\hB}{\hat{\bB}}
\newcommand{\hL}{\hat{\bL}}

\newcommand{\mydet}[1]{\det\left[#1\right]}
\newcommand{\D}[1]{ D\left( #1 \right)}
\newcommand{\myD}[1]{ \tilde{D}\left( #1 \right)}

\newcommand{\col}[2]{\mathrm{col}_{#1}\left[ #2 \right]}

\newcommand{\E}[2]{\mathbb{E}_{#1}\left[ #2 \right] }
\renewcommand{\P}[2]{\mathbb{P}_{#1}\left( #2 \right)}

\newcommand{\AND}{\quad \text{and} \quad}
\newcommand{\CH}{\mathbbm{H}}

\newcommand{\Tr}{\operatorname{Tr}}
\newcommand{\1}{\vec{1}}
\begin{document}

\title
{Improved bounds in Weaver and Feichtinger Conjectures}

\author{Marcin Bownik}

\address{Department of Mathematics, University of Oregon, Eugene,
OR 97403--1222, USA}

\email{mbownik@uoregon.edu}

\author{Pete Casazza}

\address{Department of Mathematics, University of Missouri-Columbia, Columbia, MO 65211, USA}

\email{casazzap@missouri.edu}

\author{Adam W. Marcus}

\address{Department of Mathematics, Princeton University,
Princeton, NJ 08544--1000, USA}

\email{amarcus@princeton.edu}

\author{Darrin Speegle}

\address{Department of Mathematics and Computer Science, Saint Louis 
University, 221 N. Grand Blvd.,
St. Louis, MO 63103, USA}

\email{speegled@slu.edu}

\date{\today}
%\subjclass[2000]{Primary:  Secondary: }
%\keywords{}

\thanks{The first author was partially supported by NSF grant DMS-1265711. The second author was supported by
NSF DMS 1609760; NSF ATD 1321779; and AFOSR:  FA9550-11-1-0245.
The fourth author was partially supported by a grant from the
Simons Foundation \#244953. The work on this paper was initiated during the AIM workshop ``Beyond Kadison--Singer: paving and consequences'' on December 1-5, 2014.}

%%%%%%%%%%%%%%%%

\begin{abstract} 
We sharpen the constant in the $KS_2$ conjecture of Weaver \cite{We} that was given by Marcus, Spielman, and Srivastava \cite{MSS} in their solution of the Kadison--Singer problem. 
We then apply this result to prove optimal asymptotic bounds on the size of partitions in the Feichtinger conjecture.
\end{abstract}
\maketitle

%%%%%%%%%%%%%%%

\section{Introduction} 
The goal of this paper is to explore some consequences of the recent resolution \cite{MSS} of the Kadison--Singer problem \cite{KS}. 
The Kadison--Singer problem was known to be equivalent to a large number of problems in analysis such as the Anderson Paving Conjecture \cite{An1, An2, An3}, Bourgain--Tzafriri Restricted Invertibility Conjecture \cite{BT1, BT2, BT3}, Akemann--Anderson Projection Paving Conjecture \cite{AA}, Feichtinger Conjecture \cite{BS, CCLV, Gr}, $R_\epsilon$ Conjecture \cite{CT}, and Weaver Conjecture \cite{We}. 
For an extensive study of problems equivalent to the Kadison--Singer problem we refer to \cite{CE, CFTW, CT}.
Consequently, the breakthrough resolution of the Weaver Conjecture \cite{We} by Marcus, Spielman, and Srivastava \cite{MSS} automatically validates all of these conjectures. 
At the same time, it raises the question of finding optimal quantitive bounds in these problems.

In this paper we shall concentrate on showing quantitative bounds in Weaver and Feichtinger Conjectures. 
The first part of the paper focuses on improving bounds to the conjecture of Weaver known as $KS_2$. The proof of the $KS_2$ conjecture relies on the following result \cite[Theorem 1.4]{MSS}.

\begin{theorem}[Marcus--Spielman--Srivastava]\label{thm:old}
If $\epsilon > 0$ and $v_1, \dots, v_m$ are independent random vectors in $\C^d$ with finite support such that 
\[
\sum_{i=1}^m \E{}{v_{i} v_{i}^{*}} = \bI
\AND
\E{}{ \|v_{i}\|^{2}} \leq \epsilon
\]
for all $i$, then 
\[
\P{}{\left\| \sum_{i=1}^m v_i v_i^{*} \right\| \leq (1 + \sqrt{\epsilon})^2} > 0.
\]
\end{theorem}

We show the following sharpening of Theorem \ref{thm:old}.

\begin{theorem}\label{thm:new}
If $0 < \epsilon < 1/2$ and $v_1, \dots, v_m$ are independent random vectors in $\C^d$ with support of size $2$ such that 
\begin{equation}\label{new1}
\sum_{i=1}^m \E{}{v_{i} v_{i}^{*}} = \bI
\AND
\E{}{\| v_{i}\|^{2}} \leq \epsilon 
\end{equation}
for all $i$, then 
\begin{equation}\label{new2}
\P{}{ \left\| \sum_{i=1}^m v_i v_i^{*} \right\| \leq 1 + 2\sqrt{\epsilon}\sqrt{1 - \epsilon} } > 0.
\end{equation}
\end{theorem}

Theorem \ref{thm:new} leads to improved bounds in the conjecture of Weaver known as $KS_2$. Corollary \ref{KS2} improves the original methods of \cite{MSS} that yield the same result albeit for constants $\eta>(2+\sqrt{2})^2 \approx 11.6569$.

\begin{corollary}\label{KS2}
For every $\eta>4$, there exist $\theta > 0$ such that the following holds.
Let $u_1, \dots, u_m\in \C^d$ be such that $\| u_i \| \leq 1$ for all $i$ and
\begin{equation}\label{KS3}
\sum_{i = 1}^m | \langle u, u_i \rangle |^2 = \eta
\qquad\text{for all }||u||=1.
\end{equation}
Then there exists a partition of $[m]:=\{1,\ldots,m\}$ into sets $I_1$ and $I_2$ so that for $k =1, 2$,
\begin{equation}\label{KS4}
\sum_{i \in I_k} | \langle u, u_i \rangle |^2 \leq \eta - \theta
\qquad\text{for all }||u||=1.
\end{equation}
\end{corollary}

In the second part of the paper we shall deduce quantitative bounds for the Feichtinger conjecture. 
As a consequence of Corollary \ref{KS2} we show that any Parseval frame $\{v_i\}_{i\in I} \subset \mathcal H$ (or more generally a Bessel sequence with bound $1$)  with norms $\|v_i\| \geq \ve$, where $\ve>\sqrt{3}/2$, can be decomposed into two Riesz sequences. We also show the following asymptotic estimate on the size of the partition as $\ve$ approaches to $0$.  

\begin{theorem}\label{main2} 
Suppose $\{v_i\}_{i\in I}$ is a Bessel sequence for a separable Hilbert space $\mathcal{H}$ with bound $1$ that consists of vectors of norms $\|v_i\| \geq \ve$, where $\ve>0$. 
Then there exists a universal constant $C>0$, such that $I$ can be partitioned into $r \leq C/\ve^2$ subsets $I_1, \ldots, I_r$, such that each subfamily $\{v_i\}_{i\in I_j}$, $j=1,\ldots, r$, is a Riesz sequence.
\end{theorem}

It is easy to see that Theorem \ref{main2} gives the optimal asymptotic behavior on the size $r$ of the partition. Indeed, it suffices to consider the union of $r= \lfloor 1/\ve^2 \rfloor$ copies of an orthogonal basis of $\mathcal H$ scaled by a factor $\ve$. This yields a Bessel sequence with bound $1$ that can not be partitioned into fewer than $r$ Riesz sequences.

%%%%%%%%%%%%%

\subsection{Review of Marcus, Spielman, Srivastava}\label{review}

We give a brief review of the proof in \cite{MSS} to provide some context for the statement of our main technical theorem.
See Section~\ref{sec:proof} for a more detailed discussion.

The results in \cite{MSS} use a construction introduced by the same authors in \cite{MSS0} that they called an {\em interlacing family} of polynomials. 
In \cite{MSS0}, they showed that each interlacing family constructed from a collection of polynomials $\{ p_i \}$ provides a polynomial $p^*$ with the following properties:
\begin{enumerate}
\item $p^*$ has all real roots,
\item $\mathrm{maxroot}(p^*) > \mathrm{maxroot}(p_j)$ for some $j$.
\end{enumerate}
As a result, if one can bound the largest root of the associated $p^*$, then one can assert that some polynomial in the collection has a largest root which satisfies the same bound.

In order to apply this to matrices, \cite{MSS} uses characteristic polynomials.
They consider certain convex combinations of these polynomials, which they call {\em mixed characteristic polynomials}. 
To bound the largest root of the mixed characteristic polynomial, they define a process on multivariate polynomials which starts at a determinantal polynomial and ends at a polynomial which is a multivariate version of the mixed characteristic polynomial.
They use what they call a {\em barrier function} to maintain an upper bound on the size of the largest root as the process evolves.

The bound that was proved in \cite{MSS} holds for mixed characteristic polynomials in general. 
In the first part of this paper, we consider the special case of when they are (at most) quadratic in each of its variables (corresponding to matrices of rank at most $2$).
Our main technical theorem is the following:
\begin{theorem}\label{thm:mixed}
Suppose $A_1, \ldots, A_m$ are $d \times d$ Hermitian positive semidefinite matrices of rank at most $2$ satisfying
\[
\sum_{i=1}^m A_i = \bI 
\AND 
\operatorname{Tr}(A_i) \le \epsilon<1/2 \quad\text{for all i}.
\]
Then the largest root of the polynomial
\begin{equation}\label{eqn:mixed1}
\left(\prod_{i=1}^{m} 1 - \partial_{y_i} \right) 
\mydet{\sum_{i=1}^{m} y_{i} A_{i}}
\bigg|_{y_{1} = \dots = y_{m} = x}
\end{equation}
is at most $1 + 2 \sqrt{\epsilon}\sqrt{1 - \epsilon}$.
\end{theorem}
This is an improvement over the value of $(1+\sqrt{\epsilon})^2$ in \cite[Theorem~5.1]{MSS}, but only in the case of rank $2$ matrices.
The proof follows the general outline in \cite{MSS}, but employs tighter inequalities that exploit the bounded rank of the matrices in (\ref{eqn:mixed1}).
Our main analytic tool will be the {\em mixed discriminant}, a multilinear generalization of the determinant function.
In Section \ref{sec:md} we will review the properties of the mixed discriminant that we will need in later sections.
Some of these properties are well known (see, for example \cite{Bapat, Gur}), but for the benefit of the reader we will try to make the presentation self-contained.

%%%%%%%%%%%%

\subsection{Organization}
The paper is organized as follows. 
In Section \ref{sec:md} we present some elementary properties of mixed discriminants and then in Section~\ref{sec:apply} we show how these properties can be used to establish bounds on the barrier function discussed in the previous section.
Section~\ref{sec:proof}, in particular, contains the proofs of Theorem~\ref{thm:new}, Corollary \ref{KS2}, and Theorem~\ref{thm:mixed}.

In Sections \ref{sec:frame_partitions} and \ref{sec:asymptotics}, we apply Theorems \ref{thm:old} and \ref{thm:new} to get quantitative bounds in frame theory.
Section \ref{sec:frame_partitions} contains the results from frame theory that show the interlinking properties of complementary subsets of a Parseval frame.  
In Section \ref{sec:asymptotics} we use the results in the previous sections to explore implications in frame theory. 
In particular, we prove Theorem \ref{main2} and some of its variants such as the $R_\epsilon$ conjecture and Bourgain--Tzafriri conjecture.
Our focus here will be in optimizing the bounds that follow from the results of the previous sections.

%%%%%%%%%%%%%%%

\section{Mixed discriminant and properties}\label{sec:md}

Let $S_d$ denote the symmetric group on $d$ elements.
Given $d \times d$ matrices $\bX_1, \dots, \bX_d$ and a permutation $\sigma \in S_d$, let $Y_\sigma(\bX_1, \dots, \bX_d)$ be the matrix with
\[
 \col{j}{Y_\sigma(\bX_1, \dots, \bX_n)} = \col{j}{\bX_{\sigma(j)}}
\]
where $\mathrm{col}_j$ denotes the ``$j$th column'' function.
$Y_\sigma$ can be seen as a ``mixture'' of its input matrices since each of its columns comes from a different input matrix.
\begin{definition}
\label{def:md}
The {\em mixed discriminant} of $\bX_1, \dots, \bX_d$ is the quantity
\[
\D{\bX_1, \dots, \bX_d} 
= 
\sum_{\sigma \in S_d} \mydet{Y_\sigma(\bX_1, \dots, \bX_d)}.
\]
\end{definition}
\begin{remark}\label{re:md}
Note that our definition of the mixed discriminant differs by a factor of $d!$ from many other treatments (including \cite{Bapat}), corresponding to an average over $S_d$ rather than a sum.
The literature is far from standard in this respect, and our reason for taking this normalization is that it will simplify a number of the formulas we will use.
\end{remark} 

To ease notation slightly, given a matrix $\bX$, we will write $\bX[k]$ to denote a vector that repeats $\bX$ $k$ times.
The following two examples follow directly from Definition \ref{def:md}.

\begin{example}
\label{ex:diagonal}
For $\bX$ a $d \times d$ matrix, 
\[
\D{ \bX[d] } = \D{ \bX, \bX, \dots, \bX } = d! \mydet{\bX}.
\]
\end{example}

\begin{example}
\label{ex:mytr}
For any $d \times d$ matrix $X$, 
\[
\D{\bX, \bI[d-1]} = (d-1)!\Tr(\bX).
\]
\end{example}

It should be clear that $\D{}$ is symmetric in its arguments (we will refer to this property as {\em permutation invariance}).
One property of the mixed discriminant that gives it much of its versatility as an analytic tool is its multilinearity (linearity in each input matrix).
Due to permutation invariance, it suffices to state this as linearity in the first coordinate.
\begin{lemma}[Multilinearity]
\label{lem:multilinearity}
\[
\D{a\bA + b\bB, \bX_2, \dots, \bX_d} = a\D{\bA, \bX_2, \dots, \bX_d} + b \D{\bB, \bX_2, \dots, \bX_d}
\]
\end{lemma}
\begin{proof}
It suffices to show that 
\[
\mydet{Y_\sigma(a\bA + b\bB, \bX_2, \dots, \bX_d)} = a\mydet{Y_\sigma(\bA, \bX_2,  \dots, \bX_d)} + b\mydet{Y_\sigma(\bB, \bX_2, \dots, \bX_d)}
\]
for each permutation $\sigma$ as then the same holds for the sum.
However, this follows easily from the definition and the linearity of the determinant with respect to columns:
\[
 \mydet{\begin{array}{c|c|c|c}
  a \vec{u} + b\vec{v} & \vec{x}_2 & \dots & \vec{x}_d
\end{array}
}
=
a \mydet{\begin{array}{c|c|c|c}
  \vec{u} & \vec{x}_2 & \dots & \vec{x}_d
\end{array}
}
+ 
b \mydet{\begin{array}{c|c|c|c}
  \vec{v} & \vec{x}_2 & \dots & \vec{x}_d
\end{array}
}
\]
since exactly one column of $Y_\sigma$ comes from any one of its inputs.
\end{proof}

One useful corollary of multilinearity is that the mixed discriminant has an expansion similar to the binary expansion 
\[
(a + b)^d = \sum_i \binom{d}{i} a^i b^{d-i}.
\]
Starting with Example~\ref{ex:diagonal} and iterating Lemma~\ref{lem:multilinearity} gives an analogous formula:
\begin{example}
\[
d!\mydet{xA + yB} = \D{(xA + yB)[d]} = \sum_i \binom{d}{i}x^i y^{d-i} \D{A[i], B[d-i]}.
\]
\end{example}

Mixed discriminants also have useful multiplicative properties, which are not as easily inferred from Definition~\ref{def:md}.
For this reason, we find it worthwhile to derive an equivalent characterization (which one often sees given as the primary definition).

\begin{lemma}\label{lem:def2}
For $d \times d$ matrices $\bX_1, \dots, \bX_d$, we have
\[
 \D{\bX_1, \dots, \bX_d} = \frac{\partial^d}{\partial t_1 \dots \partial t_d}  \mydet{ \sum_{i=1}^d t_i \bX_i }.
\]
\end{lemma}
\begin{proof}
Note that $\mydet{ \sum_{i=1}^d t_i \bX_i }$ is a homogeneous degree $d$ polynomial. By the linearity of the determinant with respect to columns we have
\[
\mydet{ \sum_{i=1}^d t_i \bX_i } = \sum_{i_1=1}^d \ldots \sum_{i_d=1}^d t_{i_1} \cdots t_{i_d} \mydet{\col{1}{X_{i_1}} | \ldots | \col{d}{X_{i_d}}}.
\]
Since
\[
\frac{\partial^d(t_{i_1} \cdots t_{i_d})}{\partial t_1 \dots \partial t_d} =
\begin{cases} 1 & \text{if $(i_1,\ldots,i_d)$ is a permutation of $[d]$},\\
0 & \text{otherwise,}
\end{cases}
\]
the partial derivative $\frac{\partial^d}{\partial t_1 \dots \partial t_d}$ will pick up only terms corresponding to a permutation of $[d]$. Hence,
\[
\frac{\partial^d}{\partial t_1 \dots \partial t_d}  \mydet{ \sum_{i=1}^d t_i \bX_i } = \sum_{\sigma \in S_d} \mydet{\col{1}{X_{\sigma(1)}} | \ldots | \col{d}{X_{\sigma(d)}}}
=  \D{\bX_1, \dots, \bX_d}
\]
as required.
\end{proof}

The characterization in Lemma~\ref{lem:def2} is often easier to work with than Definition \ref{def:md}.
This is evident in the following example:

\begin{example}
\label{ex:zero}
Let $\bX_1, \dots, \bX_d$ be $d \times d$ matrices and $\vec{v}$ a vector such that $\bX_k \vec{v} = \vec{0}$ for all $k$.
Then 
\[
\D{\bX_1, \dots, \bX_d} = 0
\]
\end{example}

Expressing higher rank matrices as sums of rank 1 matrices can often simplify proofs considerably.
In such cases, the following lemma is quite useful:
\begin{lemma}
\label{lem:rank1}
If $d\times d$ matrices $\bX_1, \dots, \bX_d$ have rank $1$, then
\[
\D{\bX_1, \dots, \bX_d} = \mydet{ \sum_{i=1}^d \bX_i }.
\]
\end{lemma}
\begin{proof}
Let $X_i = u_i v_i^*$.
We first note that if $v_i = v_j$ or $u_i = u_j$ for any $i \neq j$, then 
$\sum_{i=1}^d t_i u_i v_i^*$ would have rank less than $d$.
Hence by Example~\ref{ex:zero}, we have
\[
\D{\bX_1, \dots, \bX_d} = \mydet{ \sum_{i=1}^d t_i \bX_i } = 0.
\]
So assume that vectors $u_i$ (and, separately, vectors $v_i$) are distinct.
By Example~\ref{ex:diagonal}, we have
\[
\mydet{ \sum_{i=1}^d u_i v_i^* } 
=
\frac{1}{d!} \D{ \left(\sum_{i=1}^d u_i v_i^*\right)[d] } 
= 
\sum_{i_1=1}^d\dots \sum_{i_d=1}^d \D{ u_{i_1}v_{i_1}^*, \dots, u_{i_d}v_{i_d}^*}
\]
where any term with two indices $i_{j} = i_{k}$ contributes $0$ (for the reason mentioned above).
Hence the only contributing terms happen if $(i_1,\ldots,i_d)$ is a permutation of $[d]$ and there are $d!$ of these.
Furthermore, each of these gives the same contribution, since $\D{}$ is symmetric in its arguments.
Hence, we have
\[
\mydet{ \sum_i u_i v_i^* } = \frac{1}{d!} \D{ \left( \sum_i u_i v_i^* \right) [d] } = \D{ u_1v_1^*, \dots, u_d v_d^*}.
\]
\end{proof}

The next lemma is an extension of the familiar multiplication identity of the determinant $\mydet{AB} = \mydet{A}\mydet{B}$.

\begin{lemma}[Multiplication] 
\label{lem:mult}
For $d \times d$ matrices $\bX_1, \dots, \bX_d, \bY_1, \dots, \bY_d$, we have
\[
\D{\bX_1, \dots, \bX_d}\D{\bY_1, \dots, \bY_d} = \sum_{\pi \in S_d} \D{\bX_1\bY_{\pi(1)}, \dots, \bX_d\bY_{\pi(d)}}
\]
\end{lemma}
\begin{proof}
Using the characterization in Lemma~\ref{lem:def2}, we can write
\begin{align*}
\D{\bX_1, \dots, \bX_d}\D{\bY_1, \dots, \bY_d}
&= 
\frac{\partial^d}{(\partial t_1) \dots (\partial t_d)} \frac{\partial^d}{(\partial s_1) \dots (\partial s_d)} \mydet{ \sum_{i=1}^d t_i \bX_i } \mydet{ \sum_{i=1}^d s_i \bY_i } \\
&= 
\frac{\partial^d}{(\partial t_1) \dots (\partial t_d)} \frac{\partial^d}{(\partial s_1) \dots (\partial s_d)} \mydet{ \sum_{i=1}^d \sum_{j=1}^d t_i s_j \bX_i\bY_j}.
\end{align*}
Expanding the determinant using multilinearity will result in a homogeneous polynomial of degree $2d$ in the variables $s_1, \dots, s_d, t_1, \dots, t_d$, where each term will be of the form 
\[
\D{\bX_{i_1}\bY_{j_1}, \dots, \bX_{i_d}\bY_{j_d}} s_{i_1}\dots s_{i_d}t_{j_1} \dots t_{j_d} .
\]
The lemma then follows by noticing that the coefficients $\D{\bX_1\bY_{\pi(1)}, \dots, \bX_d\bY_{\pi(d)}}$ for some permutation $\pi$
are exactly the ones that will remain after the differentiations.
\end{proof}
The following examples are basic applications of Lemma~\ref{lem:mult}:
\begin{example}
\label{ex:det}
\[
\begin{aligned}
d! \mydet{\bA} \D{\bX_1, \dots, \bX_d}
&= \D{\bA[d]}\D{\bX_1, \dots, \bX_d}
\\ 
&= d! \D{\bA\bX_1, \dots, \bA\bX_d} 
= d! \D{\bX_1\bA, \dots, \bX_d\bA}.
\end{aligned}
\]
\end{example}
\begin{example}
\label{ex:tr}
\[
\begin{aligned}
\D{\bA, \bI[d-1]}\D{\bX_1, \bX_2, \dots, \bX_d} 
&= (d-1)! \bigg( \D{\bA\bX_1, \bX_2, \dots, \bX_d} \\
& + \D{\bX_1, \bA\bX_2, \dots, \bX_d} + \dots + \D{\bX_1, \bX_2, \dots, \bA\bX_d} \bigg)
\end{aligned}
\]
\end{example}

We now extend Definition~\ref{def:md} slightly so as to ease notation even further.
\begin{definition}
For $d \times d$ matrices $\bX_1, \dots, \bX_k$ where $k \leq d$, we will write
\[
\myD{\bX_1, \dots, \bX_k} = \frac{\D{\bX_1, \dots, \bX_k, \bI[d-k]}}{(d-k)!}
\]
where $\bI$ is the $d \times d$ identity matrix.
\end{definition}
In particular, note that $\D{}$ and $\myD{}$ are equivalent when there are $d$ matrices and that $\myD{\varnothing} = 1$.
Using Example~\ref{ex:det} and Example~\ref{ex:tr} above, we get the following two corollaries:

\begin{corollary}
\label{shift}
Let $\bA, \bX_1, \dots, \bX_k$ be $d \times d$ matrices with $k \leq d$ and with $\bA$ positive semidefinite.
Then
\[
\myD{\bA\bX_1, \dots, \bA\bX_k} = \myD{\bA^{1/2}\bX_1\bA^{1/2}, \dots, \bA^{1/2}\bX_k\bA^{1/2}}.
\]
\end{corollary}
\begin{proof}
\begin{align*}
\myD{\bA\bX_1, \dots, \bA\bX_k} 
&= \frac{1}{(d-k)!} \D{\bA\bX_1, \dots, \bA\bX_k, \bI[d-k]} \\
&= \frac{1}{(d-k)!} \mydet{\bA^{1/2}}\D{\bA^{1/2}\bX_1, \dots, \bA^{1/2}\bX_k, \bA^{-1/2}\bI[d-k]} \\
&= \frac{1}{(d-k)!} \D{\bA^{1/2}\bX_1\bA^{1/2} , \dots, \bA^{1/2}\bX_k\bA^{1/2}, \bA^{-1/2}\bI\bA^{1/2}[d-k]} \\
&= \myD{\bA^{1/2}\bX_1\bA^{1/2}, \dots, \bA^{1/2}\bX_k\bA^{1/2}}
\end{align*}
\end{proof}

\begin{corollary}
\label{expand}
If $\bA, \bX_1, \dots, \bX_k$ are $d \times d$ matrices with $k \leq d$, then 
\[
\myD{\bA} \myD{\bX[k]} = k \myD{\bA\bX, \bX[k-1]} + \myD{\bA, \bX[k]}.
\]
\end{corollary}
\begin{proof}
By definition, 
\[
\myD{\bA} \myD{\bX[k]} 
= \frac{1}{(d-k)!}\frac{1}{(d-1)!} \D{\bA, \bI[d-1]} \D{\bX[k], \bI[d-k]}
\]
and note that there are $k(d-1)!$ permutations $\pi$ where $\pi(1) \leq k$ and $(d-k)(d-1)!$ others.
So by Lemma~\ref{lem:mult}, we have
\begin{align*}
\D{\bA, \bI[d-1]} \D{\bX[k], \bI[d-k]} 
&= k(d-1)! \D{\bA\bX, \bX[k-1], \bI[d-k]} \\
&\qquad+ (d-k)(d-1)! \D{\bA, \bX[k], \bI[d-k-1]}
\end{align*}
and combining the two gives
\begin{align*}
\myD{\bA} \myD{\bX[k]}  
&= \frac{k}{(d-k)!}\D{\bA\bX, \bX[k-1], \bI[d-k]} \\
&\qquad+ \frac{1}{(d-k-1)!}\D{\bA, \bX[k], \bI[d-k-1]} \\
&= k \myD{\bA\bX, \bX[k-1]} + \myD{\bA, \bX[k]}
\end{align*}
as claimed.
\end{proof}
Lastly, we will require two straightforward inequalities:

\begin{lemma}[Positivity]
If $\bX_1, \dots, \bX_k$ are $d \times d$ positive semidefinite matrices with $k \leq d$, then 
\[
\myD{\bX_1, \dots, \bX_k} \geq 0.
\]
\end{lemma}
\begin{proof}
Note that it suffices to prove this when $k = d$ since 
\[
\myD{\bX_1, \dots, \bX_k} = \frac{1}{(d-k)!} \D{\bX_1, \dots, \bX_k, \bI[d-k]}
\]
and so the left hand side is nonnegative exactly when the right hand side is.
However the case when $X_1, \dots, X_d$ have rank 1 follows from Lemma~\ref{lem:rank1}, and then the general case follows by multilinearity.
\end{proof}
 
The following fact is a special case of a result due to Artstein-Avidan, Florentin, and Ostrover \cite[Theorem 1.1]{afo}.

\begin{lemma}
\label{lem:submult}
If $\bA$ and $\bB$ are $d \times d$ positive semidefinite matrices with $d \geq 2$, then
\[
\myD{\bA}\myD{\bB} \geq \myD{\bA, \bB}.
\]
\end{lemma}
\begin{proof}
By Lemma~\ref{lem:mult}, we have 
\begin{align*}
((d-1)!)^2 \myD{\bA}\myD{\bB}
&= \D{\bA, \bI[d-1]}\D{\bB, \bI[d-1]} 
\\&= (d-1)!\D{\bA\bB, \bI[d-1]} + (d! - (d-1)!)\D{\bA, \bB, \bI[d-2]}
\\&= ((d-1)!)^2\Tr(\bA\bB) + (d-1)(d-1)!\D{\bA, \bB, \bI[d-2]}
\\&= ((d-1)!)^2\Tr(\bA\bB) + ((d-1)!)^2\myD{\bA, \bB}.
\end{align*}
Rearranging gives
\[
\myD{\bA}\myD{\bB} - \myD{\bA, \bB} = \Tr(\bA\bB)
\]
which is nonnegative whenever $\bA$ and $\bB$ are positive semidefinite.
\end{proof}

Finally, we would like to mention a recent characterization of mixed discriminants by  Florentin, Milman, and Schneider \cite[Theorem 2]{fms}. Up to a multiplicative constant, the mixed discriminant is the unique function on $d$-tuples of positive semidefinite matrices that is multilinear, non-negative, and which is zero if two of its arguments are proportional matrices of rank one. 

%%%%%%%%%%%%%%%%

\section{Application to polynomials}\label{sec:apply}

Let $\CH \subset \C = \{ z : \Im(z) > 0 \}$ (where $\Im$ denotes the ``imaginary part'').
A polynomial $p \in \C[x_1, \dots, x_m]$ is called {\em stable} if $\vec{y} \in \CH^{m}$ implies $p(\vec{y}) \neq 0$.
A polynomial is called {\em real stable} if it is stable and all of its coefficients are real.

The connection between mixed discriminants and real stable polynomials can be derived from an incredibly useful result of Helton and Vinnikov \cite{HV}.  
Here, we will use an extension that specializes to our case of interest \cite[Corollary 6.7]{BB}:
\begin{theorem}\label{thm:HV}
Let $p(x, y)$ be a degree $d$ real stable polynomial.
Then there exist $d \times d$ real symmetric matrices $\bA, \bB, \bC$ such that
\[
p(x, y) = \pm \mydet{x\bA + y\bB + \bC}.
\]
Furthermore, $\bA$ and $\bB$ can be taken to be positive semidefinite.
\end{theorem}

\begin{remark}
We should note that the representation via real symmetric matrices in Theorem~\ref{thm:HV} is actually quite a bit stronger than is needed for the results in this paper.
For our purposes, it would suffice to have a representation using Hermitian matrices (a far weaker constraint, both theoretically and computationally, see \cite{HVcompute}).
\end{remark}

We would like to understand the behavior of a given real stable polynomial at a selected reference point $\vec{z}$.
Recall the following definition from \cite{MSS}:
\begin{definition} 
Let $Q(x_1,\ldots,x_m)$ be a multivariate polynomial. 
We say that a reference point $\vec{z}\in\R^m$ is {\em above the roots} of $Q$ if 
\[
Q(\vec{z}+\vec{s})>0\qquad\textrm{for all}\qquad \vec{s}=(s_1,\ldots,s_m)\in\R^m, s_i\ge 0,
\]
i.e., if $Q$ is positive on the nonnegative orthant with origin at $\vec{z}$.
\end{definition}

In the case that a reference point is above the roots of a polynomial, a more specific version of Theorem~\ref{thm:HV} can be obtained.

\begin{corollary}
\label{cor:PSD}
Let $p(x, y)$ be a degree $d$ real stable polynomial with $(x_0, y_0)$ above the roots of $p$.
Then there exist $d \times d$ real symmetric matrices $\bA, \bB, \bC$ such that 
\[
p(x, y) = \mydet{x\bA + y\bB + \bC}
\]
and the following hold:
\begin{itemize}
 \item $\bA$ and $\bB$ are positive semidefinite
 \item $\bA + \bB$ is positive definite
 \item $\bM = x_0\bA + y_0\bB + \bC$ is positive definite
\end{itemize}
\end{corollary}
\begin{proof}
 Let 
 \[
  p(x, y) = \pm \mydet{x\bA + y\bB + \bC}
 \]
be the representation provided by Theorem~\ref{thm:HV}.
Now let 
\[
 q(t) = p(t, t) = \pm \mydet{t(\bA + \bB) + \bC}.
\]
Since $p$ has (total) degree $d$, $q$ must have degree $d$ (in $t$) and so we must have $\mydet{\bA + \bB} \neq 0$.
Given that $\bA$ and $\bB$ are each positive semidefinite, this ensures $\bA + \bB$ is positive definite.
Furthermore, since $(x_0, y_0)$ is above the roots of $p$, we have $q(t) > 0$ for all for $t \geq \max \{ x_0, y_0 \} $.
Hence $q$ must have a positive first coefficient, which means
\[
 p(x, y) = \mydet{x\bA + y\bB + \bC}.
\]
Thus it remains to show that $\bM = x_0\bA + y_0\bB + \bC$ is positive definite.
To see this, consider the matrices
\[
 \bM_t = (x_0 + t)\bA + (y_0 + t)\bB + \bC.
\]
Since $\bA + \bB$ is positive definite, $\bM_t$ is positive definite for large enough $t$.
Now note that for $t \geq 0$, we have 
\[
 \mydet{\bM_t} = p(x_0 + t, y_0 + t) > 0
\]
since $(x_0, y_0)$ is above the roots of $p$.
This implies that the minimum eigenvalue of $\bM_t$ (which is a continuous function in $t$) remains above $0$ for all $t \geq 0$, and so (in particular) $\bM=\bM_0$ is positive definite.
\end{proof}

For the remainder of the section, we will fix a degree $d$ real stable polynomial $p(x, y)$ and a reference point $(x_0, y_0)$ above the roots of $p$.
We also fix the matrices $\bA, \bB, \bC$ provided by Corollary~\ref{cor:PSD} and set $\bM = x_0\bA + y_0\bB + \bC$.
Since $\bM$ is positive definite, it has a well defined square root, and so we can define the matrices
\[
\hA = \bM^{-1/2} \bA \bM^{-1/2}
\AND
\hB = \bM^{-1/2} \bB \bM^{-1/2}.
\]
Note that $\hA$ and $\hB$ are both positive semidefinite, since $\bA, \bB$ and $\bM$ are.

\begin{corollary}\label{cor:derivatives}
The matrices $\hA$ and $\hB$ defined above satisfy
\[
\frac{\partial^i}{(\partial x)^i}\frac{\partial^j}{(\partial y)^j} p(x_0, y_0) =  \myD{ \hA[i], \hB[j] } p(x_0, y_0)
\]
for all $i + j \leq d$. 
\end{corollary}
\begin{proof}
Let 
\[
p(x_0 + \epsilon, y_0 + \delta) = \sum_{i=0}^d \sum_{j=0}^{d-i} \frac{\epsilon^i}{i!} \frac{\delta^j}{j!} c_{i,j}
\]
be the Taylor expansion of a polynomial $p$ of degree $d$.
Our goal is to show that $c_{i,j} = \myD{ \hA[i], \hB[j] } p(x_0, y_0)$.
By Example~\ref{ex:diagonal}, we can write
\[
p(x_0 + \epsilon, y_0 + \delta) = \mydet{(x_0 + \epsilon)\bA + (y_0 + \delta)\bB + \bC} = \frac{1}{d!}\D{ \left(\bM + \epsilon \bA + \delta \bB \right) [d]}
\]
and using multilinearity, we have
\[
p(x_0 + \epsilon, y_0 + \delta) = \frac{1}{d!} \sum_{i = 0}^d \sum_{j=0}^{d-i} \epsilon^i \delta^j \binom{d}{i, j, d-i-j}\D{\bA[i], \bB[j], \bM[d-i-j]},
\]
where $\binom{d}{i,j, d-i-j}$ is the multinomial coefficient $\frac{d!}{i!j!(d-i-j)!}$.
Equating coefficients then gives
\begin{align*}
c_{i,j} 
&= \frac{i! j!}{d!} \binom{d}{i, j, d-i-j} \D{\bA[i], \bB[j], \bM[d-i-j]} 
\\&= \frac{1}{(d - i - j)!}\D{\bA[i], \bB[j], \bM[d-i-j]}.
\end{align*}
By Corollary~\ref{shift}, we can factor out the $\bM$ term to get
\[
c_{i,j} 
= 
\frac{1}{(d - i - j)!}\mydet{\bM}\D{\hA[i], \hB[j], \bI[d-i-j]}
= 
\mydet{M} \myD{\hA[i], \hB[j]}
\]
which, since $\mydet{\bM} = p(x_0, y_0)$, is exactly the claimed result.
\end{proof}

Corollary~\ref{cor:derivatives} provides a way to associate the partial derivatives of a bivariate real stable polynomial at a point to a mixed discriminant involving the matrices in its determinantal representation.
This, coupled with Corollary~\ref{cor:PSD}, will allow us to use properties of positive semidefinite matrices to derive inequalities for points above the roots of $p$.
With this in mind, we will attempt to quantify the concept of being ``above the roots''.
For a polynomial $q$, we will consider the barrier function
\[
\Phi_q^x = 
\begin{cases}
\frac{\partial}{\partial x} \ln q(x_0, y_0) & \text{when $(x_0, y_0)$ is above the roots of $q$} \\
\infty & \text{otherwise}
\end{cases}
\]
and $\Phi_q^y$ defined similarly (except with the derivative in the $y$ coordinate).
In particular, we will be interested in the behavior of these functions under transformations of $q$.
The next lemma is a general result in that direction:

\begin{lemma}
\label{ifandonlyif}
Let $R = \sum_i a_i \partial_x^i$ be a differential operator with real coefficients $\{ a_i \}$ such that $q = R(p)$ and such that $(x_0, y_0)$ is above the roots of both $p$ and $q$.
Then the following two statements are equivalent:
\begin{enumerate}
\item $\Phi_{q}^y \leq \Phi_p^y$
\item $\sum_i i a_i  \myD{\hA[i-1], \hL} \geq 0$, where $\hL = \hA^{1/2}\hB\hA^{1/2}$ is positive semidefinite.
\end{enumerate}
\end{lemma}
\begin{proof}
Using Corollary~\ref{cor:derivatives}, we can write
\[
\frac{q}{p} = \sum_i a_i \myD{\hA[i]}
\]
so that 
\[
\frac{q_y}{p} = \sum_i a_i \myD{\hA[i], \hB}
\AND
\frac{q p_y}{p^2} = \sum_i a_i \myD{\hA[i]}\myD{\hB}.
\]
Now using Corollary~\ref{expand}, we have
\[
\myD{\hA[i]}\D{\hB} = \myD{\hA[i], \hB} + i \myD{\hA[i-1], \hA \hB}.
\]
By Corollary~\ref{cor:PSD}, both $\hA$ and $\hB$ are positive semidefinite, and so we can apply Corollary~\ref{shift} to get 
\[
\myD{\hA[i], \hA \hB} = \myD{\hA[i], \hL}
\]
where $\hL = \hA^{1/2}\hB\hA^{1/2}$ is positive semidefinite.
Combining these gives
\[
\frac{p q_y - q p_y}{p^2} = \sum_i a_i \left( \myD{\hA[i], \hB} - \myD{\hA[i]}\myD{\hB} \right) = - \sum_i i a_i  \myD{\hA[i-1], \hL}.
\]
Therefore we can write 
\[
\Phi_{q}^y - \Phi_p^y 
=
\frac{p q_y - q p_y}{pq} 
= 
\left(\frac{p}{q} \right) \frac{p q_y - q p_y}{p^2}
= 
-\left(\frac{p}{q} \right) \sum_i i a_i  \myD{\hA[i-1], \hL}
\]
where $p/q$ is positive when $(x_0, y_0)$ is above the roots of both $p$ and $q$.
Hence
\[
\Phi_{q}^y - \Phi_p^y 
\AND
- \sum_i i a_i  \myD{\hA[i-1], \hL}
\]
have the same sign, as required.
\end{proof}

Using this machinery, we now prove two lemmas that will help us exploit the quadratic nature of the polynomials.
The first of these lemmas is a strengthening of \cite[Lemma~5.11]{MSS} in the case that the polynomial is quadratic.

\begin{lemma}\label{lem:phi}
Assume $p(x, y)$ is quadratic in $x$ and let 
\[
\Phi_p^x \leq \left(1 - \frac{1}{\delta} \right) \frac{1}{2 - \delta}
\]
for some $\delta \in (1, 2)$.
Now let $q(x, y) = (1 - \partial_x) p(x + \delta, y)$ and assume that $(x_0, y_0)$ is above the roots of both $p$ and $q$.
Then
\[
\Phi_{q}^y \leq \Phi_{p}^y.
\]
\end{lemma}
\begin{proof}
We first write $q$ as $R(p)$ for a differential operator $R$.
The shift by $\delta$ can be translated into a differential operator using Taylor's formula:
\begin{equation}\label{taylor}
f(x + t) = \sum_i f^{(i)}(x) \frac{t^i}{i!} = \sum_i \frac{t^i}{i!} \partial_x^i f.
\end{equation}
Hence we can write
\begin{align*}
q(x, y) &=
p(x + \delta, y) - p_x(x + \delta, y) \\
&= \left(1 + \delta \partial_x + \frac{\delta^2}{2} \partial_x^2\right) p(x, y) - \left(\partial_x + \delta \partial_x^2 \right) p(x, y) \\ 
&= \left(1 + (\delta - 1) \partial_x + \left(\frac{\delta^2}{2} - \delta \right) \partial_x^2 \right) p
\end{align*}
where all higher level derivatives can be discarded since $p$ is quadratic.
Hence we can write
\[
q(x, y) = a_0 p + a_1 p_x + a_2 p_{xx}
\]
where $a_0 = 1$ and 
\begin{equation}\label{eq:a1a2}
a_1 =  \delta - 1
\AND
a_2 = \frac{\delta^2}{2} - \delta.
\end{equation}
Then Lemma~\ref{ifandonlyif} implies
\[
\Phi_{q}^y \leq \Phi_p^y
\]
if and only if 
\begin{equation}
\label{eq:blah}
a_1 \myD{\hL} + 2 a_2 \myD{\hA, \hL} \geq 0.
\end{equation}
Since $\hL$ is positive semidefinite, $\myD{\hL} = 0$ if and only if $\hL$ is the $0$ matrix.
In this case, we would also have $\myD{\hA, \hL} = 0$ and so (\ref{eq:blah}) would hold trivially (thus finishing the proof).
On the other hand, if $\myD{\hL} \neq 0$ then (\ref{eq:blah}) holds if and only if 
\[
a_1 \geq -2 a_2 \frac{\myD{\hA, \hL}}{\myD{\hL}} \geq 0.
\]
By Corollary~\ref{cor:PSD}, $\hA$ and $\hL$ are positive semidefinite, so by Lemma~\ref{lem:submult} we have the inequality 
\[
\myD{\hA, \hL} \leq \myD{\hA} \myD{\hL} =  \myD{\hL} \Phi_p^x.
\]
since $\myD{A} = \Phi_p^x$ by Corollary~\ref{cor:derivatives}.
Now since $\delta \in (1, 2)$, we have $a_1 > 0$ and $a_2 < 0$, and so it suffices to show
\[
a_1  \geq -2 a_2 \Phi_p^x.
\]
Plugging back in the values for $a_1$ and $a_2$ from (\ref{eq:a1a2}) gives
\[
\delta - 1 \geq \delta (2-\delta)\Phi_p^x 
\]
which is precisely our initial hypothesis.
\end{proof}

In Section~\ref{sec:actualproof}, we will use Lemma~\ref{lem:phi} to understand how the transformation from $q$ to $p$ changes the value of the barrier function $\Phi$.
To do so, however, we will need to ensure that the point $(x_0, y_0)$ is above the roots of the resulting $q$ (something that is not true in general).
In \cite{MSS}, this was addressed (for an appropriately chosen point $(x_0, y_0)$) using \cite[Lemma~5.10]{MSS}.
Again we will need a strengthened version that takes advantage of the quadratic nature of our polynomials.

\begin{lemma}\label{lem:decreasing} % Fixed by ADAM, per referee
Let $s$ be a quadratic, univariate polynomial with positive first coefficient and real roots $a \leq b$.
Then
\[
f(x) = x - \frac{2s(x)}{s'(x)}
\]
is a nonincreasing function for any $x > b$.
\end{lemma}
\begin{proof}
We start by writing
\[
f'(x) = 1 - 2\frac{s'(x)^2 - s''(x)s(x)}{s'(x)^2} = 2\frac{s''(x)s(x)}{s'(x)^2} - 1.
\]
By Taylor's formula, we have
\[
s(x + y) = s(x) + y s'(x) + \frac{y^2}{2} s''(x)
\]
which is real-rooted (as a polynomial in $y$) and therefore by the quadratic formula we must have
\[
s'(x)^2 \geq 2 s(x) s''(x)
\]
with equality if and only if $s$ has a double root.
Since $s$ has positive first coefficient, both $s(x)$ and $s''(x)$ are nonnegative for $x \geq b$ --- we therefore have
\[
f'(x) \leq 0
\]
with equality if and only if $s$ has a double root.
\end{proof}

%%%%%%%%%%%%%

\section{Proof of Theorem~\ref{thm:new}}\label{sec:proof}

The purpose of this section is to prove Theorem~\ref{thm:mixed} using the tools from Sections \ref{sec:md} and \ref{sec:apply}.
Theorem \ref{thm:new} can then be deduced from Theorem~\ref{thm:mixed} by the same argument as \cite[Theorem 1.4]{MSS} (which we briefly review here).

Given random vectors $v_1, \dots, v_m \in \C^d$, one can define the (random) matrix $V = \sum_i v_i v_i^*$ and its (random) characteristic polynomial
\[
 p_V(x) = \mydet{x \mathbf I - V}.
\]
In the case that the random vectors $\{ v_i \}$ are independent, the authors of \cite{MSS} constructed a so-called {\em interlacing family} from the polynomials in the support of $p_V$.
As mentioned in Section~\ref{review}, any such construction provides a polynomial $p^*$ with the following properties:
\begin{enumerate}
\item $p^*$ has all real roots,
\item $\P{}{\mathrm{maxroot}(p_{V}) < \mathrm{maxroot}(p^*)} > 0$.
\end{enumerate}
Since each $p_V$ is the characteristic polynomial of a positive semidefinite matrix, the largest root of $p_V$ is the operator norm of $V$.
Hence a conclusion like the one in Theorem~\ref{thm:new} could be obtained by finding an appropriate bound on $\mathrm{maxroot}(p^*)$.

In the case of the interlacing family constructed in \cite{MSS}, the associated polynomial $p^*$ is the expected characteristic polynomial
\begin{equation}\label{eq:expected_poly}
 p^*(x) = \E{}{ p_V(x)} = \E{}{\mydet{x\mathbf I - \sum_i v_i v_i^*}}.
\end{equation}
As the first step in the process of bounding the largest root of $p^*$, Marcus, Spielman, and Srivastava \cite[Theorem 4.1]{MSS} showed that (\ref{eq:expected_poly}) could be written in a form they call a {\em mixed characteristic polynomial}:
\[
 p^*(x) = \E{}{\mydet{x \mathbf I - \sum_i v_i v_i^*}} 
= \left(\prod_{i=1}^{m} 1 - \partial_{y_i} \right) 
\mydet{\sum_{i=1}^{m} y_{i} A_{i}}
\bigg|_{y_{1} = \dots = y_{m} = x}
\]
where $A_i = \E{}{v_i v_i^*}$ for each $i$.
By translating the restrictions on the $v_i$ in the hypothesis of Theorem~\ref{thm:new} to restrictions on the associated $A_i$, one can see that Theorem~\ref{thm:mixed} is precisely the bound on the largest root of $p^*$ necessary for Theorem~\ref{thm:new}.

Once Theorem~\ref{thm:new} is established, the deduction of Corollary \ref{KS2} then follows the same proof as \cite[Corollary~1.5]{MSS}. 
For the benefit of the reader, we reproduce it here:
\begin{proof}[Proof of Corollary~\ref{KS2}]
For each $i\in [m]$ and $k=1,2$, we define vectors $w_{i,k} \in \C^{2d}$ by
\[
w_{i,1}= \sqrt{2/\eta} \begin{bmatrix} u_i \\
0^d \end{bmatrix},
\
w_{i,2}= \sqrt{2/\eta} \begin{bmatrix} 0^d \\
u_i \end{bmatrix}.
\]
Let $v_1,\ldots,v_m$ be independent random vectors such that $\P{}{v_i=  w_{i,k}}=1/2$, $k=1,2$. 
A simple calculation using \eqref{KS3} shows that these vectors satisfy \eqref{new1} with $\epsilon=2/\eta<1/2$. 
By Theorem \ref{thm:new}, there exists an assignment of each $v_i$ so that
\[
\left\| \sum_{i=1}^m v_i v_i^{*} \right\| =
\left\| \sum_{k=1}^2 \sum_{i\in S_k} w_{i,k} w_{i,k}^{*} \right\|
\leq 1 + 2\sqrt{\epsilon}\sqrt{1 - \epsilon},
\]
where $I_k=\{i: v_i=w_{i,k} \}$. Hence, for $k=1,2$, we have
\[
\left\| \sum_{i\in S_k} u_i u_i^{*} \right\| =\frac{\eta}{2} 
\left\| \sum_{i\in S_k} w_{i,k} w_{i,k}^{*} \right\|
\leq \frac{\eta}{2} \bigl(1 + 2\sqrt{\epsilon(1 - \epsilon)}\bigr)=\frac{\eta}{2}+\sqrt{2(\eta-2)}.
\]
This shows \eqref{KS4} with $\theta=\eta/2-\sqrt{2(\eta-2)}>0$ (when $\eta > 4$).
\end{proof}

\subsection{Proof of Theorem~\ref{thm:mixed}}\label{sec:actualproof}

We first recall the formal definition of the {\em barrier function} from \cite{MSS}:

\begin{definition}
\label{def:barrier}
For a polynomial $Q \in \C[x_1, \dots, x_m]$ and a point $\vec{y} \in \R^m$, we define the {\em barrier function of $Q$ at $\vec{y}$} to be the function 
\[
\Phi^i_Q(\vec{y}) = 
\begin{cases}
\frac{\partial}{\partial x_i} \ln Q(\vec{y}) & \text{when $\vec{y}$ is above the roots of $Q$} \\
\infty & \text{otherwise}
\end{cases}
\]
\end{definition}
This is an extension of the function $\Phi$ introduced in the previous section to allow for different reference points and more variables (the coordinates $x, y$ have been replaced by variables $x_i$, and only the subscript $i$ is used so as to reduce the clutter).

Let $\epsilon < 1/2$ and set
\[
\delta = 1 + \sqrt{\frac{\epsilon}{1 - \epsilon}}
\AND
t = (1 - 2 \epsilon) \sqrt{\frac{\epsilon}{1 - \epsilon}}.
\]
We start with a polynomial $Q_0$ and reference point $w_0$ defined as
\[
Q_0 = \mydet{\sum_i x_i A_i}
\AND
w_0 = t \1
\]
where $\1$ is the vector with $d$ $1$'s.
It is easy to check that $w_0$ is above the roots of $Q_0$ and that $\Phi_{Q_0}^j(w_0) \leq \frac{\epsilon}{t}$ for all $j$ as in the proof of \cite[Theorem 5.1]{MSS}.

Given $Q_i$ and $w_i$, we will construct polynomial $Q_{i+1}$ and reference point $w_{i+1}$ as
\[
Q_{i+1} = (1 - \partial_{i+1}) Q_i 
\AND 
w_{i+1} = w_i + \delta e_{i+1}.
\]
Our goal is to show that $w_{i+1}$ is above the roots of $Q_{i+1}$.
To do so, we will need to understand the effect of applying the $(1 - \partial_{i+1})$ operator on the polynomial $Q_i$, for which we can use the barrier function $\Phi_{Q_{i}}^{i+1}(w_i)$:

\begin{lemma}\label{lem:b2}
If $w_i$ is above the roots of $Q_i$ and 
\[
\Phi_{Q_i}^{i+1}(w_i) \leq \frac{\epsilon}{t}
\]
then 
$w_{i+1}$ is above the roots of $Q_{i+1}$.
\end{lemma}
\begin{proof}
Let $s(x)$ be the univariate polynomial that comes from holding all variables of $Q_i$ other than the $(i+1)$st variable constant, i.e., $s(x)=Q_i(t+\delta,\ldots,t+\delta,x,t,\ldots,t)$.
By the monotonicity of barrier functions \cite[Lemma 5.8]{MSS}, it suffices to show that 
\begin{equation}\label{eqn:plus_delta}
s(t + \delta) - s'(t + \delta) > 0
\end{equation}
given that 
\[
\Phi_s(t) := \frac{s'(t)}{s(t)} \leq \frac{\epsilon}{t}.
\]
Note that equation~(\ref{eqn:plus_delta}) is equivalent to $\Phi_s(t + \delta) < 1$.
By Lemma~\ref{lem:decreasing}, we have
\[
t - \frac{2}{\Phi_s(t)} \geq t +\delta - \frac{2}{\Phi_s(t+\delta)}
\]
and so 
\[
\frac{1}{\Phi_s(t + \delta)} \geq \frac{\delta}{2} + \frac{1}{\Phi_s(t)}.
\]
Thus it suffices to show 
\[
\frac{\delta}{2} + \frac{1}{\Phi_s(t)} > 1
\]
which is equivalent to showing
\[
\Phi_s(t) < \frac{2}{2 - \delta}.
\]
Plugging in the hypothesis, it suffices to show 
\[
\frac{\epsilon}{t} < \frac{2}{2 - \delta}
\]
which reduces to showing
\[
\frac{ \sqrt{\epsilon}\sqrt{1 - \epsilon}}{1 - 2 \epsilon} < \frac{2}{1 - \sqrt{\frac{\epsilon}{1 - \epsilon}}}
\]
when the given values of $\delta$ and $t$ are inserted.
It is then easy to check that this holds for any $\epsilon < 1/2$.
\end{proof}

In order to use Lemma~\ref{lem:b2}, we will need to bound the value of $\Phi_{Q_i}^{i+1}(w_i)$.
We will do this by showing that the transformation from $(Q_i, w_i)$ to $(Q_{i+1}, w_{i+1})$ causes the barrier functions to shrink in all coordinates $j > i+1$.
Note that when moving from $(Q_i, w_i)$ to $(Q_{i+1}, w_{i+1})$, we are altering only the $x_{i+1}$ variable.
Hence to see what happens to the barrier function in coordinate $j > i+1$, we can restrict to those two coordinates (since the restriction of a real stable polynomial is a real stable polynomial) and appeal to Lemma~\ref{lem:phi}.
\begin{lemma}\label{lem:b1}
Let $j$ be a coordinate such that $i+1 < j \leq m$ and
\[
\Phi_{Q_i}^{j}(w_i) \leq \frac{\epsilon}{t}.
\]
Then
\[of
\Phi_{Q_{i+1}}^{j}(w_{i+1}) \leq \frac{\epsilon}{t}.
\]
\end{lemma}
\begin{proof}
Lemma~\ref{lem:b2} ensures that $w_{i+1}$ is above the roots of $Q_{i+1}$ and so it is sufficient (by Lemma~\ref{lem:phi}) to show 
\[
\Phi_{Q_i}^{j}(w_i) \leq \left(1 - \frac{1}{\delta}\right)\frac{1}{2 - \delta}.
\]
Using the hypothesis, this would be implied by showing
\[
\frac{\epsilon}{t} \leq \left(1 - \frac{1}{\delta}\right)\frac{1}{2 - \delta}.
\]
However one can easily check that 
\[
\frac{\epsilon}{t} = \frac{ \sqrt{\epsilon}\sqrt{1 - \epsilon}}{1 - 2 \epsilon} = \left(1 - \frac{1}{\delta}\right)\frac{1}{2 - \delta}
\]
and so we are done.
\end{proof}

Iterating $m$ times, Lemmas \ref{lem:b2} and \ref{lem:b1} ensure that $w_m$ is above the roots of $Q_m$, where $Q_m$ is exactly the polynomial in Equation~(\ref{eqn:mixed1}) before the variables are set to $x$.
Furthermore, $w_m = (t + \delta)\1$.
Hence $w_m$ being above the roots of $Q_m$ implies the largest root of Equation~(\ref{eqn:mixed1}) is at most 
\[
t + \delta = 1 + 2 \sqrt{\epsilon}\sqrt{1 - \epsilon}
\]
as required for Theorem \ref{thm:mixed}. 

\begin{remark}
The argument here is more delicate than the one given in \cite{MSS}; this can be seen by comparing the statement of Lemma~\ref{lem:phi} to its analogous version \cite[Lemma~5.11]{MSS}.
In \cite{MSS}, any $\delta$ that caused the barrier function to contract also resulted in $w_{i+1}$ being above the roots of $Q_{i+1}$ (a fortiori).
This is not the case here and is the reason that the additional hypothesis of $(x_0, y_0)$ being above the roots is necessary in Lemma~\ref{lem:phi}.
This also becomes evident when considering the space of values $(\delta, t)$ for which Lemma~\ref{lem:b2} and Lemma~\ref{lem:b1} hold. 
In \cite{MSS}, the constraint provided by \cite[Lemma~5.11]{MSS} was the only relevant one in determining the optimal values of $\delta$ and $t$, whereas in our case both Lemma~\ref{lem:phi} and Lemma~\ref{lem:decreasing} provide nontrivial constraints.
\end{remark}

%%%%%%%%%%%%%%

\section{Naimark's complements of frame partitions} \label{sec:frame_partitions}

In this section we establish a result that links properties of two complementary subsets of a Parseval frame with the corresponding subsets of the Naimark's complement. In general, a subset of a Parseval frame does not need to be a frame, and we can only expect it to be a Bessel sequence with bound $1$. In Proposition \ref{ce} we show that if this subset has a Bessel bound strictly less than $1$, then its corresponding Naimark's complement subset is a Riesz sequence.

It is rather surprising that Proposition \ref{ce} has not appeared in the frame theory literature before despite its simplicity and the elementary nature of its proof. However, it can be considered as a quantitative variant of the complementarity principle between spanning and linear independence due to Bodmann, Casazza, Paulsen, and Speegle \cite[Proposition 2.3]{BCPS}. We start with basic conventions in frame theory \cite{Ch}.

\begin{definition}\label{def1}
A family of vectors $\{\phi_i \}_{i\in I}$ in a Hilbert space
$\mathcal H$ is called a {\it frame} for $\mathcal H$ if
there are constants $0<A\le B < \infty$ (called 
{\it lower and upper frame bounds}, respectively)
so that 
\begin{equation}\label{E5}
A\|\phi\|^2 \le \sum_{i\in I}|\langle \phi,\phi_i \rangle |^2
\le
B\|\phi\|^2 \qquad\text{for all }\phi\in \mathcal H.
\end{equation}
If we only have the right hand inequality in \eqref{E5},
 we call $\{\phi_i\}_{i\in I}$
a {\it Bessel sequence} with Bessel bound $B$.
  If $A=B$, $\{\phi_i \}_{i\in I}$ is called a
{\it tight frame} and if $A=B=1$, it is called a
{\it Parseval frame}.
\end{definition}

\begin{definition}\label{def2}
 A family of vectors $\{\phi_i\}_{i\in I}$ in a Hilbert space $\mathcal H$
is a {\it Riesz sequence} if there are constants $A,B>0$ so that
for all $\{a_i\}\in \ell^2(I)$ we have
\begin{equation}\label{R7}
A\sum_{i\in I}|a_i |^2 \le \bigg\|\sum_{i\in I}a_i \phi_i \bigg\|^2 \le
B\sum_{i\in I}|a_i|^2.
\end{equation}
We call $A,B$ {\it lower and upper Riesz
bounds} for $\{\phi_i\}_{i\in I}$.
\end{definition}

Note that it suffices to verify \eqref{R7} only for sequences $\{a_i\}$ with finitely many non-zero coefficients, since a standard convergence argument yields the same bounds \eqref{R7} for all infinitely supported sequences $\{a_i\}\in \ell^2(I)$.
In general we do not require that frame, Bessel, and Riesz bounds in Definitions \ref{def1} and \ref{def2} are optimal. In particular, a Bessel sequence with bound $B$ is automatically a Bessel sequence with bound $B' \ge B$.

\begin{notation}
Throughout the rest of the paper $\{e_i\}_{i\in I}$ will
denote an orthonormal basis for whatever space we are
working in.
\end{notation}

\begin{proposition}\label{ce} Let $P:\ell^2(I) \to \ell^2(I)$ be the orthogonal projection onto a closed subspace $\mathcal H \subset \ell^2(I)$. Then, for any subset $J \subset I$ and $\delta>0$, the following are equivalent:
\begin{enumerate}[(i)]
\item
$\{ P e_i \}_{i \in J}$ is a Bessel sequence with bound $1-\delta$,
\item
$\{ P e_i \}_{i \in J^c}$ is a frame with lower bound $\delta$,
\item
$\{ (\mathbf I - P) e_i \}_{i \in J}$ is a Riesz sequence with lower bound $\delta$, where $\mathbf I$ is the identity on $\ell^2(I)$.
\end{enumerate}
\end{proposition}

\begin{proof}
Since $\{P e_i\}_{i\in I}$ is a Parseval frame for $\mathcal H$,   we have
\begin{equation}\label{ce1}
\|\phi\|^2=\sum_{i\in J} |\lan \phi, Pe_i \ran |^2 + \sum_{i\in J^c} |\lan \phi, P e_i \ran |^2 \qquad\text{for all } \phi \in \mathcal H.
\end{equation}
Thus, 
\begin{equation}\label{ce2}
\sum_{i\in J} |\lan \phi, Pe_i \ran |^2 \le (1-\delta) \|\phi\|^2
\iff  \sum_{i\in J^c} |\lan \phi, P e_i \ran |^2 \ge \delta \|\phi\|^2 \qquad\text{for all } \phi \in \mathcal H.
\end{equation}
By \eqref{ce1}, the Bessel sequence $\{ P e_i \}_{i \in J^c}$ inherits the Bessel bound $1$ as a subset of a Parseval frame. Thus, \eqref{ce2} shows the equivalence of (i) and (ii).

To show the equivalence of (i) and (iii), note that for any sequence of coefficients $\{a_i\} \in \ell^2(J)$,
\begin{equation}\label{ce3}
\sum_{i\in J} |a_i|^2 = \bigg\| \sum_{i\in J} a_i P e_i \bigg\|^2 + \bigg\|\sum_{i\in J} a_i (\mathbf I - P) e_i \bigg\|^2.
\end{equation}
Thus,
\begin{equation}\label{ce4}
\bigg\| \sum_{i\in J} a_i P e_i \bigg\|^2 \le (1-\delta) \sum_{i\in J} |a_i|^2 
\iff
\bigg\| \sum_{i\in J} a_i (\mathbf I - P) e_i \bigg\|^2 \ge \delta \sum_{i\in J} |a_i|^2.
\end{equation}
By \eqref{ce3}, the family $\{ (\mathbf I - P) e_i \}_{i \in J}$ has automatically the Riesz upper bound $1$.
Observe that the inequality in left hand side of \eqref{ce4} is equivalent to (i). This follows from the well-known fact that adjoint of the analysis operator
\[
T: \mathcal H \to \ell^2(I), \qquad T \phi = \{\lan \phi, P e_i \ran\}_{i\in J}, \quad \phi \in \mathcal H,
\]
is the synthesis operator
\[
T^*: \ell^2(I) \to \mathcal H, \qquad T^*(\{a_i\}_{i\in J}) = \sum_{i\in J} a_i Pe_i, \quad \{a_i\}_{i\in J} \in \ell^2(J).
\]
Since $||T||=||T^*||$, \eqref{ce4} yields the equivalence of (i) and (iii).
\end{proof}

\begin{remark}
A curious reader might ask what condition needs to be imposed about $\{ (\mathbf I - P) e_i \}_{i \in J^c}$ to obtain the equivalence in Proposition \ref{ce}. Surprisingly, this condition can not be easily stated in terms of Bessel, Riesz, or frame bounds. Instead, it is not difficult to show that the following restricted Riesz upper bound condition does the job:
\begin{equation}\label{ce5}\tag{iv}
\bigg\| \sum_{i\in J^c} \lan \phi, Pe_i \ran (\mathbf I - P) e_i \bigg\|^2 \le (1-\delta)\sum_{i\in J^c} |\lan \phi, Pe_i \ran|^2 \qquad\text{for all } \phi \in \mathcal H.
\end{equation}
Since this observation will not be used subsequently in the paper, we leave the verification of details to the reader. 
\end{remark}

As an immediate consequence of Proposition \ref{ce} we have

\begin{corollary}\label{cek}
Let $P: \ell^2(I) \to \ell^2(I)$ be the orthogonal projection onto a closed subspace $\mathcal H \subset \ell^2(I)$. Then, for any subset $J \subset I$ and $\delta>0$, the following are equivalent:
\begin{enumerate}[(i)]
\item
$\{ P e_i \}_{i \in J}$ is a frame with frame bounds $\delta$ and $1-\delta$,
\item
$\{ P e_i \}_{i \in J^c}$ is a frame with frame bounds $\delta$ and $1-\delta$,
\item
both $\{ P e_i \}_{i \in J}$  and $\{ P e_i \}_{i \in J^c}$ are Bessel sequences with bound $1-\delta$,
\item
both $\{ (\mathbf I - P) e_i \}_{i \in J}$ and $\{ (\mathbf I - P) e_i \}_{i \in J^c}$ are Riesz sequences with lower bound $\delta$.
\end{enumerate}
\end{corollary}

\begin{proof}
Suppose that (iii) holds. Applying Proposition \ref{ce} simultaneously to Bessel sequences $\{ P e_i \}_{i \in J}$  and $\{ P e_i \}_{i \in J^c}$ yields the remaining properties (i), (ii), and (iv). Conversely, any of these properties implies (iii) in light of Proposition \ref{ce}.
\end{proof}

%%%%%%%%%%%%%%%%%

\section{Asymptotic bounds in Feichtinger Conjecture}\label{sec:asymptotics}

In this section we establish quantitative bounds in Feichtinger Conjecture. To achieve this we shall employ the results of the previous section and the landmark result of Marcus, Spielman, and Srivastava \cite[Corollary 1.5]{MSS}. In the language of Bessel sequences this result takes the following form, where $\mathcal{H}_n$ denotes $n$-dimensional real or complex Hilbert space $\mathbb{R}^n$ or $\mathbb{C}^n$.

\begin{theorem}\label{MSS}
Let $\{u_i\}_{i=1}^M \subset \mathcal H_n$ be a Bessel sequence with bound $1$ and $\|u_i\|^2 \le \delta$ for all $i$. Then for any positive integer $r$, there exists a partition $\{I_1,\ldots, I_r\}$ of $[M]$ such that each $\{u_i\}_{i\in  I_j}$, $j=1,\ldots,r$, is a Bessel sequence with bound 
\[
\left(\frac{1}{\sqrt{r}} + \sqrt{\delta}\right)^2.
\]
\end{theorem}

\begin{remark}\label{eb}
Note that the original formulation \cite{MSS} of Theorem \ref{MSS} requires that $\{u_i\}_{i=1}^M$ is Parseval frame. 
This can be relaxed since any Bessel sequence $\{u_i\}_{i=1}^M$ with bound $1$ can be extended to a Parseval frame by adding additional vectors satisfying $\|u_i\|^2 \le \delta$ for $i=M+1, \ldots,\tilde M$, $\tilde M>M$. This is a consequence of the Schur--Horn Theorem \cite{AMRS, BJ}, see also the proof of Corollary \ref{PM}.
\end{remark}

Corollary \ref{KS2} gives us a quantitative version of Weaver Conjecture $KS_2$ with sharper constants than those deducible from Theorem \ref{MSS}. In particular, a simple rescaling of Corollary \ref{KS2} combined with Remark \ref{eb} yields the following theorem, which we state for Bessel sequences.

\begin{theorem}\label{mss}
Let $0<\delta_0<1/4$ and $\ve_0=1/2-\sqrt{2\delta_0(1-2\delta_0)}$.
Suppose that $\{\phi_i\}_{i=1}^M$ is a Bessel sequence in $\mathcal H_n$ with Bessel bound 1 and that $\|\phi_i\|^2 \le \delta_0$ for all $i$. Then there exists a partition $\{I_1, I_2\}$ of $[M]$ such that each $\{\phi_i\}_{i\in  I_j}$, $j=1,2$, is a Bessel sequence with bound $1-\ve_0$.
\end{theorem}

\begin{notation}Since the above constants might potentially be improved in the future, we shall keep $\delta_0$ and $\epsilon_0$ as base parameters that shall propagate to all remaining results in this section. Consequently, we shall fix the constants $\delta_0$ and $\ve_0$ as in Theorem \ref{mss} throughout this section. In particular, combining Corollary \ref{cek} with Theorem \ref{mss} yields the following result with the same constants $\delta_0$ and $\epsilon_0$.
\end{notation}

\begin{corollary}\label{AM}
Suppose that $\{\phi_i\}_{i=1}^M$ is a Parseval frame for $\mathcal H_n$ and $\|\phi_i\|^2 \ge 1-\delta_0$ for all $i$. Then there exists a partition $\{I_1, I_2\}$ of $[M]$ such that each $\{\phi_i\}_{i\in  I_j}$, $j=1,2$, is a Riesz sequence with lower bound $\ve_0$.
\end{corollary}

\begin{proof}
By Naimark's Theorem we can write $\phi_i=P e_i$, where $P$ is an orthogonal projection of $\mathcal H_M$ onto $n$-dimensional subspace identified with $\mathcal H_n$. Applying Theorem \ref{mss} to $\{(\mathbf I - P)e_i\}_{i=1}^M$ yields a partition $\{I_1, I_2\}$ of $[M]$ such that each $\{(\mathbf I - P)e_i\}_{i\in  I_j}$, $j=1,2$, is a Bessel sequence with bound $1-\ve_0$. By Corollary \ref{cek}, each $\{Pe_i\}_{i\in I_j}$, $j=1,2$, is a Riesz sequence with lower bound $\ve_0$.
\end{proof}

Next, we extend the validity of Corollary \ref{AM} to Bessel sequences. This requires a more sophisticated approach than what was outlined in Remark \ref{eb}. 

\begin{corollary}\label{PM}
Suppose that $\{\phi_i\}_{i=1}^M$ is a Bessel sequence with bound $B$ for $\mathcal H_n$ and $\|\phi_i\|^2 \ge B(1-\delta_0)$ for all $i$. Then there exists a partition $\{I_1, I_2\}$ of $[M]$ such that each $\{\phi_i\}_{i\in  I_j}$, $j=1,2$, is a Riesz sequence with lower bound $B\ve_0$.
\end{corollary}

\begin{proof}
Without loss of generality we can assume that the Bessel bound $B=1$.
Suppose that the frame operator of $\{\phi_i\}_{i=1}^M$, which is given by
\[
S = \sum_{i=1}^M \phi_ i \phi_i^*,
\]
has eigenvalues $1\ge \lambda_1 \ge \lambda_2 \ge \ldots \ge \lambda_n \ge 0$. For a fixed $N$, consider an operator $\tilde S = \mathbf I_{n+N} - S \oplus \mathbf 0_N$ on $\mathcal H_{n+N}=\mathcal H_n\oplus \mathcal H_N$, where $\mathbf 0_N$ is the zero operator on $\mathcal H_N$. Then, $\tilde S$ has the following eigenvalues listed in decreasing order 
\begin{equation}\label{PM2}
\underbrace{1,\ldots,1}_{N},1-\lambda_n, \ldots, 1-\lambda_1.
\end{equation}
Our goal is to find a collection of vectors $\{\phi_{M+i}\}_{i=1}^K$ in $\mathcal H_{n+N}$ such that:
\begin{enumerate}[(i)]
\item its frame operator is $\tilde S$, and
\item
$\|\phi_{M+i} \|^2=C$ for all $i=1,\ldots, K$ for some constant $C \in [1-\delta_0,1]$.
\end{enumerate} 
By the Schur--Horn Theorem \cite{AMRS, BJ} this is possible if and only if the sequence \eqref{PM2} majorizes
\[
\underbrace{C,\ldots,C}_K.
\]
This, in turn, is implied by
\begin{equation}\label{PM3}
K \ge N+n \qquad\text{and}\qquad CK=N+\sum_{i=1}^n (1-\lambda_i).
\end{equation}
By choosing sufficiently large $N$ and $K=N+n$, we have that
\[
1-\delta_0 \le C=\frac{N+\sum_{i=1}^n (1-\lambda_i)}{N+n} \le 1.
\]
This shows the existence of vectors $\{\phi_{M+i}\}_{i=1}^K$ satisfying (i) and (ii).

Now, $\{\phi_i\}_{i=1}^{M+K}$ is a Parseval frame for $\mathcal H_{n+N}$ such that $\| \phi_i \|^2 \ge 1-\delta_0$ for all $i$. By Corollary \ref{AM} there exists a partition $\tilde I_1$, $\tilde I_2$ of $[M+K]$ such that each $\{\phi_i\}_{i\in  \tilde I_j}$, $j=1,2$, is a Riesz sequence with lower bound $\ve_0$. Then, $I_j= \tilde I_j \cap [M]$ is the required partition of $[M]$.
\end{proof}

We are now ready to show the asymptotic estimate on the size of a partition in the Feichtinger Conjecture.

\begin{theorem} \label{fei}
Suppose that $\{\phi_i\}_{i=1}^M$ is a Bessel sequence with bound $1$ for $\mathcal H_n$ and $\|\phi_i\|^2 \ge \ve>0$ for all $i$. Then there exists $r=r(\ve)=O(1/\ve)$ and a partition $\{I_1,\ldots, I_r\}$ of $[M]$ such that each $\{\phi_i\}_{i\in  I_j}$, $j=1,\ldots, r$, is a Riesz sequence with lower bound $\ve_0\ve$. In addition, if $\|\phi_i\|^2 = \ve$ for all $i$, then the upper Riesz bound of each $\{\phi_i\}_{i\in  I_j}$ is $\frac{\ve}{1-\delta_0}$.
\end{theorem}

\begin{proof} 
First, observe that without loss of generality we can assume that $\|\phi_i\|^2=\ve$ for all $i$. Indeed, we can replace each $\phi_i$ by $\frac{\sqrt{\ve}}{\|\phi_i\|} \phi_i$ and then apply Theorem \ref{fei} to get the general case.

By Theorem \ref{MSS} for each $\tilde r$ we can find a partition $\{\tilde I_j\}_{j=1}^{\tilde r}$ of $[M]$ such that each $\{\phi_i\}_{i\in  \tilde I_j}$  is a Bessel sequence with bound 
\[
B=\bigg(\frac{1}{\sqrt{\tilde r}} + \sqrt{\ve}\bigg)^2.
\]
Then we wish to apply Corollary \ref{PM} to each such $\{\phi_i\}_{i\in  \tilde I_j}$. This is possible if we choose $\tilde r$ such that
\begin{equation}\label{fei2}
\|\phi_i\|^2 = \ve \ge B(1-\delta_0)= \bigg(\frac{1}{\sqrt{\tilde r}} + \sqrt{\ve}\bigg)^2(1-\delta_0).
\end{equation}
A simple calculation shows that the above inequality simplifies to
\[
\frac{2}{\sqrt{\tilde r \ve}} + \frac{1}{\tilde r \ve} \le \frac{\delta_0}{1-\delta_0}.
\]
Hence, it suffices to choose 
\[
\tilde r \ge \frac{9}{\ve} \bigg(\frac{1-\delta_0}{\delta_0} \bigg)^2.
\]
By Corollary \ref{PM}, each $\{\phi_i\}_{i\in  \tilde I_j}$ can be partitioned into two 2 Riesz sequences with lower bound $B\ve_0 \ge \ve_0 \ve$ and upper bound $B\le \frac{\ve}{1-\delta_0}$.
This gives the required partition of size $r=2\tilde r$ and completes the proof of Theorem \ref{fei}.
\end{proof}

\begin{remark} Note that choosing $\tilde r$ so that \eqref{fei2} is almost an equality yields the slightly better lower bound $\frac{\ve_0}{1-\delta_0} \ve$ in the conclusion of Theorem \ref{fei}. Moreover, since we can take any $\delta_0<1/4$, we obtain an explicit estimate on the size of a partition $r=r(\ve)\le 162/\ve$. Surely, the number $162$ is far from sharp, but it merely gives a crude upper bound on optimal constant.
\end{remark}

\subsection{Infinite dimensional results}
Corollary \ref{KS2} can be easily extended to the infinite dimensional setting using the ``pinball principle" \cite[Proposition 2.1]{CCLV}, which we state here for the reader's convenience.

\begin{theorem}\label{pinball} 
Fix a natural number $r$ and assume for every natural number $n$, we have a partition $\{I_i^n\}_{i=1}^r$ of $[n]$. Then there are natural numbers $\{n_1 < n_2 < \cdots \}$ so that if $j \in I_i^{n_j}$ for some $i \in [r]$, then $j \in I_i^{n_k}$ for all $k \ge j$. For any $i \in [r]$ define $I_i = \{j: j\in I_i^{n_j}\}$. Then,
\begin{enumerate}[(i)]
\item $\{I_i\}_{i=1}^r$ is a partition of $\N$.
\item If $I_i = \{j_1 < j_2 < \cdots\}$, then for every natural number $k$ we have $\{j_1,j_2,\ldots,j_k\} \subset I_i^{n_{j_k}}.$
\end{enumerate}
\end{theorem}

\begin{theorem}
\label{main1} 
Suppose $\{\phi_i\}_{i\in I}$ is a Bessel sequence in a separable Hilbert space $\mathcal{H}$ with constant $\eta>4$, which consists of vectors of norms $\|\phi_i\|\leq 1$. That is, 
\[
\sum_{i\in I} |\langle \phi,\phi_i \rangle |^2 \le \eta \|\phi\|^2
\qquad\text{for all }\phi\in\mathcal{H}.
\]
Then there exists a constant $\theta=\theta(\eta)>0$, depending only on $\eta>4$, such that the index set $I$ can be decomposed into subsets $I_1$ and $I_2$, so that the Bessel sequences $\{\phi_i\}_{i\in I_j}$, $j=1,2$, have bounds $\le \eta-\theta$, i.e.,
\[
\sum_{i\in I_j} |\langle \phi,\phi_i \rangle |^2 \le (\eta-\theta) \|\phi\|^2
\qquad\text{for all }\phi\in\mathcal H, \ j=1,2.
\]
\end{theorem}

\begin{proof}
Without loss of generality, we can assume that the index set $I = \N$. For each $n \in \N$, the set $\{\phi_i\}_{i=1}^n$ is a Bessel sequence in $\mathcal{H}$ with Bessel bound $\eta$. By Theorem \ref{mss}, for each $n$, there exists a partition $\{I_1^n, I_2^n\}$ of $[n]$ such that for each $j = 1,2$, $\{\phi_i\}_{i\in I^n_j}$ is a Bessel sequence with bound $\eta - \theta$. Let $\{I_1, I_2\}$ be the partition of $\N$ obtained by applying Theorem \ref{pinball} to the sequence of partitions $\{I_1^n, I_2^n\}_{n=1}^\infty$. For any $j = 1,2$, we write $I_j = \{j_1 < j_2 < \cdots\}$. By Theorem \ref{pinball}(ii) for any $k\in \N$, $\{\phi_{j_1},\ldots, \phi_{j_k}\}$ is a Bessel sequence with bound $\eta - \theta$. Since $k\in\N$ is arbitrary, $\{\phi_i : i\in I_j\}$ is a Bessel sequence with bound $\eta - \theta$ for $j = 1,2$, as desired. 
\end{proof}

 In a similar way, the ``pinball principle'' implies the infinite dimensional Theorem \ref{fei-infty} from finite dimensional Theorem \ref{fei}. Theorem \ref{fei-infty} is simply a rescaled variant of Theorem \ref{main2}.

\begin{theorem} \label{fei-infty}
Suppose that $\{\phi_i\}_{i\in I} $ is a Bessel sequence with bound $B$ for a separable Hilbert space $\mathcal H$ and $\|\phi_i\| \ge 1$ for all $i$. Then there exists $r=O(B)$ and a partition $\{I_1,\ldots, I_r\}$ of $I$ such that each $\{\phi_i\}_{i\in  I_j}$, $j=1,\ldots, r$, is a Riesz sequence with lower bound $\ve_0$. Moreover, if $\|\phi_i\| = 1$ for all $i$, then the upper bound is $\frac{1}{1-\delta_0}$.
\end{theorem}

\begin{proof}
Without loss of generality we can assume that the index set $I=\mathbb N$ and $\|\phi_i\| = 1$ for all $i$. For every $n\in\N$, we apply Theorem \ref{fei} to the Bessel sequence $\{\frac1{\sqrt{B}}\phi_i\}_{i=1}^n$ to obtain a partition $\{I^n_1,\ldots, I^n_r\}$ of $[n]$ such that each subsequence $\{\frac{1}{\sqrt{B}}\phi_i\}_{i\in I_j^n}$ is a Riesz sequence with bounds $\ve_0/B$ and $\frac{1}{B(1-\delta_0)}$. The size of this partition is some fixed $r =O(B)$ as $B\to \infty$. Let $\{I_1, \ldots, I_r\}$ be the partition of $\N$ obtained by applying Theorem \ref{pinball} to the sequence of partitions $\{I_1^n,\ldots, I_r^n\}_{n=1}^\infty$. For any $j\in [r]$, we write $I_j = \{j_1 < j_2 < \cdots\}$. By Theorem \ref{pinball}(ii) for any $k\in \N$, $\{\phi_{j_1},\ldots, \phi_{j_k}\}$ is a Riesz sequence with bounds $\ve_0$ and $\frac{1}{1-\delta_0}$.
Since $k\in\N$ is arbitrary, we conclude that $\{\phi_i\}_{i\in I_j}$ is a Riesz sequence with the same bounds for every $j=1,\ldots,r$. 
\end{proof}

It is also worth investigating the size of a partition if we insist on having nearly tight Riesz sequences. That is, Riesz sequences with bounds $1-\ve$ and $1+\ve$, where $\ve>0$ is an arbitrarily small parameter. That is, we are asking for the size of partition in $R_\ve$ Conjecture that was shown by Casazza and Tremain \cite{CT} to be equivalent with the Kadison--Singer problem.

\begin{theorem} \label{fei-r}
Suppose that $\{\phi_i\}_{i\in I}$ is a unit norm Bessel sequence with bound $B$ for a separable Hilbert space $\mathcal H$. Then for any $\ve>0$ there exists a partition $\{I_1,\ldots, I_r\}$ of $I$ of size $r=O(B/\ve^4)$, such that each $\{\phi_i\}_{i\in  I_j}$, $j=1,\ldots, r$, is a Riesz sequence with bounds  $1-\ve$ and $1+\ve$.
\end{theorem}

In the proof of Theorem \ref{fei-r} we will use the following lemma. The case when $J=I$ is a well-known fact, see \cite[Section 3.6]{Ch}. For the sake of completeness we will give the proof of Lemma \ref{subriesz}.

\begin{lemma}\label{subriesz} 
Suppose $\{\phi_i\}_{i\in I}$ is a Riesz basis in a Hilbert space $\mathcal H$, and let $\{\phi^*_i\}_{i\in I}$ be its unique biorthogonal (dual) Riesz basis, i.e.,
\[
\lan \phi_i,\phi_j^* \ran =\delta_{i,j} \qquad\text{for all }i,j\in I.
\]
Then, for any subset $J\subset I$, the Riesz sequence bounds of $\{\phi_i\}_{i\in J}$ are $A$ and $B$ if and only if the Riesz sequence bounds of $\{\phi^*_i\}_{i\in J}$ are $1/B$ and $1/A$.
\end{lemma}

\begin{proof} Suppose that $\{\phi_i\}_{i\in J}$ has upper Riesz bound $B$. This is equivalent to the Bessel condition
\begin{equation}\label{sr2}
\sum_{i\in J} |\lan f, \phi_i \ran |^2 \le B ||f||^2
\qquad\text{for all }f\in \mathcal H.
\end{equation}
For any sequence $\{a_i\}_{i\in J} \in \ell^2$, there exists a unique $f\in \mathcal H$ such that
\[
\lan f, \phi_i \ran = \begin{cases} a_i & i\in J,\\
0 & \text{otherwise}.
\end{cases}
\]
Since $f= \sum_{i\in J} a_i \phi^*_i$, by \eqref{sr2} we have
\[
\bigg\| \sum_{i\in J} a_i \phi^*_i \bigg\|^2 = ||f||^2 \ge \frac{1}{B} \sum_{i\in J} |\lan f, \phi_i \ran |^2 = \frac{1}{B} \sum_{i\in J} |a_i|^2.
\]
Conversely, if $\{\phi^*_i\}_{i\in J}$ has lower Riesz bound $1/B$, then \eqref{sr2} holds and $\{\phi_i\}_{i\in J}$ has upper Riesz bound $B$. 
By symmetry, $\{\phi^*_i\}_{i\in J}$ has upper Riesz bound $1/A$ if and only if $\{\phi_i\}_{i\in J}$ has lower Riesz bound $A$, which completes the proof of the lemma.
\end{proof}

\begin{proof}[Proof of Theorem \ref{fei-r}] By the ``pinball principle'' as in the proof of Theorems \ref{main1} and  \ref{fei-infty} it suffices to restrict our attention to the finite dimensional case.

In the first step we apply Theorem \ref{fei-infty} to find a partition of size $O(B)$ into Riesz sequences with bounds $\ve_0$ and $\frac1{1-\delta_0}$. Suppose that $\{\phi_i\}_{i=1}^M$ is one of these, i.e., a unit-norm Riesz sequence with bounds $\ve_0$ and $\frac1{1-\delta_0}$. Let $\{\phi^*_i\}_{i=1}^M$ be its unique biorthogonal (dual) Riesz sequence in $\mathcal H'=\operatorname{span}\{\phi_i\}_{i=1}^M$ that has bounds $1-\delta_0$ and $\frac{1}{\ve_0}$ by Lemma \ref{subriesz}. In the second step we apply Theorem \ref{MSS} to both of these Riesz sequences to reduce their upper Riesz bounds to $1+\ve$. This requires partitions of size $\tilde r=O(1/\ve^2)$ since we need to guarantee that
\[
\bigg(\frac{1}{\sqrt{\tilde r}}+1\bigg)^2 \le 1+\ve.
\]
Now it suffices to consider a common refinement partition $I_1,\ldots, I_r$ of size $r=O(B/\ve^4)$ of a partition in the first step and two partitions in the second step. For any $j=1,\ldots, r$, $\{\phi_i\}_{i\in I_j}$ is a Riesz sequence with upper bound $1+\ve$. Since $\{\phi^*_i\}_{i\in I_j}$ is also a Riesz sequence with the same upper bound, Lemma \ref{subriesz} implies that $\{\phi_i\}_{i\in I_j}$ has lower bound $1/(1+\ve) \ge 1 - \ve$. This completes the proof of Theorem \ref{fei-r}.  
\end{proof}

\begin{remark} It is an open problem what is the optimal dependence of the size of the partition on $\ve$. The linear dependence on $B$ is optimal, but it is less clear whether one can reduce dependence on $\ve$ from $O(1/\ve^4)$ to some lower exponent. This problem is closely related with finding the optimal size of partition in Anderson's Paving Conjecture, see \cite[Theorem 6.1]{MSS}. It is known \cite[Theorem 6]{CEKP} that size of partition must be at least $O(1/\ve^2)$ as $\ve \to 0$.
\end{remark}

Repeating the standard arguments as in \cite{CCLV}, Theorem \ref{fei-infty} yields the same asymptotic bounds on the size of partition for Bourgain--Tzafriri Conjecture.

\begin{theorem} There exist universal constants $c,\ve_0, \delta_0>0$ so that for any $B>1$ the following holds. Suppose $T: \mathcal H \to \mathcal H$ is a linear operator  with norm $\|T\| \le \sqrt{B}$ and $\|Te_i\|=1$ for all $i\in I$, where $\{e_i\}_{i\in I}$ is an orthonormal basis of a separable Hilbert space $\mathcal H$. 
Then, there exists a partition $\{I_1,\ldots, I_r\}$ of the index set $I$ of size $r\le cB$, so that for all $j=1,\ldots, r$ and all choice of scalars $\{a_i\}_{i\in I_j} \in \ell^2$ we have
\[
\ve_0 \sum_{i\in I_j} |a_i|^2 \le \bigg\| \sum_{i\in I_j} a_i Te_i \bigg\|^2 \le \frac{1}{1-\delta_0} \sum_{i\in I_j} |a_i|^2.
\]
\end{theorem}

As a consequence of our results we obtain explicit bounds on the partition size for Fourier frames. If $E \subset [0,1]$ has positive Lebesgue measure, then the collection of functions $\phi_n(t)=e^{2\pi i nt} \chi_E(t)$, $n\in \Z$, is a Parseval frame for $L^2(E)$, often called a Fourier frame. Since this is an equal norm frame, i.e., $||\phi_n||^2=|E|$ for all $n\in \Z$, Theorem \ref{fei-infty} yields the following corollary. Moreover, by the results of Lawton \cite{La} and Paulsen \cite[Theorem 1.2]{Pa}, the index sets can be chosen to be a syndetic set. Recall that $I\subset \Z$ is syndetic if for some finite set $F\subset \Z$ we have
\[
\bigcup_{n\in F} (I-n) = \Z.
\]

\begin{corollary} There exists a universal constant $c>0$ such that for any subset $E \subset [0,1]$ with positive measure, the corresponding Fourier frame $\{\phi_n\}_{n\in\Z}$ can be decomposed as the union of $r \le c/|E|$ Riesz sequences $\{\phi_n\}_{n\in I_j}$, $j=1,\ldots,r$. Moreover, each index set $I_j$ can be chosen to be a syndetic set.
\end{corollary}

\bibliographystyle{amsplain}

\end{document}